\documentclass[10pt,a4paper]{article}

\usepackage{amsthm}
\usepackage{amscd}
\usepackage{euscript}
\usepackage{amsfonts}
\usepackage{amssymb}
\usepackage{enumerate}
\usepackage{epsfig}
\usepackage{caption2}
\usepackage{amsmath}
\usepackage{subfigure}

\newcommand{\CC}{\mathbb{C}}

\newcommand{\RR}{\mathbb{R}}

\newcommand{\ints}{\int\limits}

\newcommand{\lam}{\lambda}

\newcommand{\ep}{\varepsilon}

\newcommand{\OO}{\mathcal{O}}
\newcommand{\de}{\partial}

\newcommand{\zi}{\zeta}

\newcommand{\Cam}{\mathcal{C}}
\newcommand{\lsomma}{\sum\limits_{l=0}^{M}}
\newcommand{\lsommaguid}{\sum\limits_{l=1}^{M}}
\newcommand{\vp}{\varphi}

\theoremstyle{plain}
\newtheorem{theorem}{Theorem}[section]

\newtheorem{lemma}[theorem]{Lemma}

\theoremstyle{definition}

\newtheorem{remark}[theorem]{Remark}

\newenvironment{proofTeo1}{
  \noindent{\it Proof of Theorem \ref{teo1}.}\ }{\hspace*{\fill}
  \begin{math}\Box\end{math}\medskip}

\newenvironment{proofTeo2}{
  \noindent{\it Proof of Theorem \ref{teo2}.}\ }{\hspace*{\fill}
  \begin{math}\Box\end{math}\medskip}

\begin{document}

\title{A radiation condition for the 2-D Helmholtz equation in stratified media}

\author{Giulio Ciraolo \footnote{Dipartimento di Matematica e
Applicazioni, Universit\`a di Palermo, Via Archirafi 34, 90123
Palermo, Italy, ({\tt g.ciraolo@math.unipa.it}).} }

\maketitle

\begin{abstract}
We study the 2-D Helmholtz equation in perturbed stratified media,
allowing the existence of guided waves. Our assumptions on the
perturbing and source terms are not too restrictive.

We prove two results. Firstly, we introduce a Sommerfeld-Rellich
radiation condition and prove the uniqueness of the solution for the
studied equation. Then, by careful asymptotic estimates, we prove
the existence of a bounded solution satisfying our radiation
condition.
\end{abstract}

\section{Introduction} \label{section introduction}
A classical problem in studying the Helmholtz equation
\begin{equation}\label{helm}
\Delta u + k^2 n(x,z)^2 u = f, \quad (x,z) \in \RR^2
\end{equation}
is that of finding a physically meaningful criterion for uniqueness
of solutions. When $k$ is real (and nonzero) and $n$ is a
real-valued function, the Sommerfeld radiation condition (see
\cite{So1} and \cite{So2}) and the Rellich Theorem \cite{Rel} are
the basis for such studies. Many papers have been written to extend
the Sommerfeld and Rellich radiation conditions to situation in
which the index of refraction $n$ has special properties. If the
refraction index tends to a constant $n_\infty$ in all directions
(with an appropriate behaviour), the usual uniqueness assumption is
given by the so-called outgoing Sommerfeld radiation condition
\begin{equation} \label{rad cond sommerfeld}
\lim_{R\to +\infty} R^{\frac{N}{2}} (u_R - ikn_{\infty} u) = 0,
\end{equation}
uniformly; here, $N$ is the dimension of the space and $R$ is the
radial variable. Under the same assumptions, Rellich condition is
\begin{equation} \label{rad cond Rellich}
\lim_{R\to +\infty} \ints_{\de B_R} |u_R - ikn_{\infty} u|^2 d\sigma
= 0,
\end{equation}
where $B_R$ is the ball of radius $R$ and $d\sigma$ is the surface element.

Both the conditions above say something about the geometry of the
level sets of the phase of the solution: they are circles at leading
order and their radii grow at a specific rate.

When such homogeneity condition at infinity of the refraction index
is perturbed, it is unclear which should be the right geometry. Many
papers have been written on this topic; we refer to Section 1 in
\cite{CM2} and references therein for a more detailed description of
known results. Moreover, {\it large} perturbations of some fixed
refraction index could change the rate of growing of the radii of
the level sets of the phase function (that correspond to the right
choice of the propagation constant in the radiation condition).

In this paper a step forward in those directions is given for the
Helmholtz equation
\begin{equation}\label{helm perturbata}
\Delta u + [k^2 n(x)^2 + p(x,z)] u = f, \quad (x,z) \in \RR^2,
\end{equation}
where $n$ is of the form
\begin{equation}\label{n}
  n:= \begin{cases} n_+, & x > h, \\ n_{co}(x), & |x|\leq h,\\
  n_-, & x < - h; \end{cases}
\end{equation}
here $k>0$, $n_{co}(\cdot)$ is a real-valued function of bounded
variation of the variable $x$, with $n_+, n_- , h $ positive
constants and $p$ is a perturbing term satisfying certain hypothesis
to be specified later.

Our work is motivated by the study of infinite open waveguides.Under
the \emph{weakly guiding approximation} (see \cite{SL}), and for
$p\equiv 0$, \eqref{helm perturbata} describes the electromagnetic
wave propagation in an optical or acoustical waveguide, where $k$ is
the wavenumber and $n$ is the index of refraction. The peculiarity
of the problem is the fact that the index of refraction $n$ is not a
compact or \emph{small} perturbation of the plane, and it may cause
the appearance of {\it guided modes}, i.e. waves which propagates
(each one with a different constant of propagation) in the
$z$-direction without decaying. We will call {\it radiating waves}
the waves that are not guided by the waveguide.

Our work is based on the knowledge of a Green's function $G$ for the
non-perturbed ($p\equiv 0$) Helmholtz equation
\begin{equation}\label{Helm rectilinear}
\Delta u + k^2 n(x)^2 u = f, \quad (x,z)\in\RR,
\end{equation}
with $n$ given by \eqref{n}; as done in \cite{Wi}, such a Green's
function can be found by using Titchmarsh theory \cite{Ti} on
eigenfunction expansions. We will make use of the results and
notations in \cite{MS},\cite{CM1},\cite{CM2},\cite{Ci2} (where the
case $n_+=n_-$ is deeply studied) and \cite{Ci1} and \cite{Ch} (for
the expression of the Green's function in the general case).

Due to the presence of guided modes, the usual Sommerfeld radiation
condition does not guarantee the uniqueness of solutions. The
conditions proposed in \cite{CM2}, \cite{Xu1} and \cite{Xu2} provide
the uniqueness for the Helmholtz equation in stratified media and
they consist in a collection of Sommerfeld-like conditions for all
guided components of the field and for the radiative component, each
of them having its own wavenumber. In \cite{Xu1},\cite{Xu2}, the
author studies the case of a stratified medium with compactly
supported inhomogeneities and gives a radiation condition in the
spirit of \eqref{rad cond sommerfeld}. In \cite{CM2} and \cite{Ci2}
analogous results are obtained by using an integral formulation of
the radiation condition. In the present paper, we improve the
mentioned results in the following sense: (i) we weaken the
radiation condition (we use a radiation condition which is in the
spirit of \eqref{rad cond Rellich}); (ii) we consider
inhomogeneities that can be extended to infinity in the direction of
the waveguide but have to be small in some sense (see (H2) later).

We denote by $u_0$ the radiated part of the solution,
$u_1,\ldots,u_M$ the guided ones and $\beta_l$ the propagation
constant corresponding to $u_l$, $l=0,1,\ldots,M$, (see \cite{CM2}
or Section \ref{section uniqueness} for a rigorous definition of
$u_l$ and $\beta_l$). Then, the radiation condition introduced in
\cite{CM2} (for the case $n_{cl}:=n_+=n_-$) is
\begin{equation}\label{rad cond CM 1}
\ints_0^\infty \ints_{\partial \Omega_R} \Big{|} \frac{\partial
u_0}{\partial \nu} - i k n_{cl} u_0  \Big{|}^2 d\ell \, d R +
\sum\limits_{l=1}^M \ints_0^\infty \ints_{\partial Q_R} \Big{|}
\frac{\partial u_l}{\partial \nu} - i\beta_l u_l  \Big{|}^2 d\ell \,
d R < + \infty,
\end{equation}
where $R=\sqrt{x^2+z^2}$, $\nu$ denotes the outward normal
derivative and
\begin{equation} \label{omega rho}
Q_R= \left\{ (x,z) \in \RR^2 : |x|, |z| \leq R \right\}, \ \
\Omega_R = \left\{ (x,z) \in \RR^2 : [x]_h^2 + z^2 \leq R^2
\right\},
\end{equation}
with
\begin{equation}\label{xh quadra} [x]_h=
\begin{cases}
x+h, & x<-h, \\
0, & -h \leq x \leq h, \\
x-h, & x> h.
\end{cases}
\end{equation}
In \cite{CM2} it was also noticed that also the following radiation
condition
\begin{equation}\label{rad cond CM 2}
    \lsomma \ints_0^\infty \ints_{\partial \Omega_R}
    \Big{|} \frac{\partial u_l}{\partial \nu} - i\beta_l u_l
    \Big{|}^2 d\ell \, d R < + \infty,
\end{equation}
still guarantees the existence and uniqueness of a solution for
\eqref{helm perturbata}.

Both condition \eqref{rad cond CM 1} and \eqref{rad cond CM 2} say
that the level sets of the phase of the radiating part of the
solution are given by the sets $\de \Omega_R$. An asymptotical
approximation of the sets $\de \Omega_R$ may be also used in the
radiation conditions (in particular, it may be also a ball); we
prefer to use the sets $\Omega_R$ because they lighten the analysis
the asymptotic behaviour of the Green's function (see \cite{CM2}).
Regarding guided waves, \eqref{rad cond CM 1} and \eqref{rad cond CM
2} do not seem to distinguish which is the right geometry of the
level sets, even if both of them ensure the uniqueness of the
problem. We notice that guided waves are one dimensional solutions
of the Helmholtz equation and thus the level sets of the phase
function are just straight lines in the $x$-direction.

In this paper we provide a radiation condition of Sommerfeld-Rellich
type which guarantees the uniqueness of solutions of \eqref{helm
perturbata}, with $n$ given by \eqref{n} and where $p:\RR^2 \to \CC$
is such that
\begin{enumerate}[(H1)]
\item $p(x,z)=0$ for $|x| > x_0$ for some positive $x_0$;
\item $p$ satisfies
\begin{equation}\label{p condizioni}
\sup_{(\xi,\zi) \in \RR^2} \ints_{\RR^2} |G(x,z;\xi,\zi) p(x,z)| dx
dz < 1;
\end{equation}
\end{enumerate}
here, $G$ is the Green's function for the unperturbed stratified
medium mentioned above (see Section \ref{section preliminaries} for
more details). In particular, we are assuming that the perturbation
is small in some sense and has compact support in the direction
transversal to the waveguide. Our first result is the following:

\begin{theorem} \label{teo1}
Let $p$ satisfy assumptions (H1) and (H2). There exists at most one
bounded solution of \eqref{helm perturbata} satisfying
\begin{equation}\label{rad cond}
   \lim_{R\to +\infty} \ints_{\partial \Omega_R}
    \Big{|} \frac{\partial u_0}{\partial \nu} - ikn(x) u_0
    \Big{|}^2 d\ell + \lsommaguid \sqrt{R} \ints_{\partial Q_R}
    \Big{|} \frac{\partial u_l}{\partial \nu} - i\beta_l u_l
    \Big{|}^2 d\ell =0;
\end{equation}
here, $\nu$ denote the outward normal and $\Omega_R$ and $Q_R$ are
given by \eqref{omega rho}.
\end{theorem}

We notice that Theorem \ref{teo1} still holds if we consider the
following radiation condition
\begin{equation}\label{rad cond_II}
\lim_{R\to +\infty} \lsomma \ints_{\partial \Omega_R} \Big{|}
\frac{\partial u_l}{\partial \nu} - i\beta_l u_l \Big{|}^2 d\ell =0.
\end{equation}
We prefer to use \eqref{rad cond} because it better describes the
behaviour of guided modes: (i) the sets $Q_R$  suggest the geometry
of the level sets of the guided modes (which are straight lines in
the $x$-direction); (ii) the presence of $\sqrt{R}$ suggests that
guided modes are lower dimensional solution of the Helmholtz
equation.

\vspace{1em}

The latter result of this paper concerns the existence of a solution
satisfying \eqref{rad cond}. We shall assume that $p$ satisfies the
following additional assumption:
\begin{enumerate}[(H3)]
\item $p \in L^2(\RR^2)$ is such that
\begin{equation}\label{int vp}
\ints_{\de \Omega_R} |p|^2 d\ell \leq c_1 R^{-(3+2\delta)},
\end{equation}
for some constant $c_1>0$ and $\delta>\frac{1}{2}$.
\end{enumerate}
Then, our result is the following:

\begin{theorem} \label{teo2}
Let $f$ and $p$ satisfy the assumptions (H1) and (H3) and assume
that $p$ satisfies (H2), too. Then, there exist a unique bounded
solution of \eqref{helm perturbata} satisfying the radiation
condition \eqref{rad cond}.

In particular, such a solution is the only bounded solution of the
following integral equation:
\begin{equation}\label{eq integrale u}
u(x,z) = \ints_{\RR^2} G(x,z;\xi,\zi)
[f(\xi,\zi)-p(\xi,\zi)u(\xi,\zi)] d\xi d\zi.
\end{equation}
\end{theorem}

We notice that, if the waveguide is not rectilinear, the propagation
constants $\beta_l$ become complex (see, for instance, \cite{KNH}).
Theorem \ref{teo2} guarantees that, under the given assumptions, the
propagation constants of the radiating and guided parts of the
solution are (approximately) the same as in the unperturbed case and
\eqref{rad cond} still guarantee the existence and uniqueness of a
solution. To the author's knowledge, it is not known if the exponent
$\delta$ in (H3) can be improved (see also \cite{Ei} for the case of
non-stratified medium).

\vspace{1em}

\noindent The paper is organized as follows.

In Section \ref{section
preliminaries}, we recall and prove some preliminary results which
will be useful in the rest of the paper.

Theorem \ref{teo1} will be proved in Section \ref{section
uniqueness}. The technique used is in the spirit of classical
results on the Helmholtz equation, in particular those contained in
\cite{Mi1} and \cite{Mi2}. Other techniques may be used to prove
such theorem (for instance, the Limiting Absorption Principle, see
\cite{Ho} and \cite{We}); we shall include our proof of Theorem
\ref{teo1} because it is simple and direct.

In Section \ref{section existence} we will prove Theorem \ref{teo2}.
Here, a careful analysis of the asymptotic behaviour of the solution
is done. Similar arguments for the free space case can be found in
\cite{Ei}.

We wish to mention that our approach can be generalized to
stratified media in higher dimensions and in more general unbounded
domains. Clearly, stratified media in higher dimensions may present
more than one kind of stratification (in three dimensions, for
instance, planar or cylindrical stratifications lead to different
behaviours of the solution). Once a uniform asymptotic expansion of
the Green's function is known, then it is possible to use the same
technique in this paper and obtain analogous results. This will be
the object of future work.

\section{Preliminaries} \label{section preliminaries}
In this section we recall and prove some results for the unperturbed
Helmholtz equation, which will be useful in the rest of the paper.
We notice that the case $n_+=n_-$ has been deeply studied in
\cite{MS},\cite{CM1},\cite{CM2} and \cite{Ci2} and we refer to such
works for a more extensive description of results and of the
formulation of the outgoing Green's function.

By following \cite{Wi} (see also Chapter 2 in \cite{Ci1}), we write
a solution $u$ of \eqref{Helm rectilinear} in terms of a Green's
function $G$, which is a superposition of solutions of the
associated homogeneous equation:
\begin{equation} \label{u}
u(x,z) = \ints_{\RR^2} G(x,z;\xi,\zi) f (\xi,\zi) d\xi d\zi,
\end{equation}
where
\begin{equation} \label{G G^rad+somma G_l^g}
G(x,z;\xi,\zi)= G_0(x,z;\xi,\zi) + \sum_{l=1}^M G_l (x,z;\xi,\zi),
\end{equation}
with
\begin{equation}\label{G l guid}
G_l(x,z;\xi,\zi) = \frac{e^{i\beta_l |z-\zi|}}{2i \beta_l}
e(x,\gamma_l) e(\xi,\gamma_l),\quad l=1,\ldots,M,
\end{equation}
and
\begin{equation}\label{beta_l}
\beta_l= \sqrt{k^2 n_*^2 - \gamma_l},\quad l=1,\ldots,M.
\end{equation}
Here, $\gamma_l$ and $e(x,\gamma_l)$, $l=1,\ldots,M$, are,
respectively, the eigenvalues and eigenfunctions of the eigenvalue
problem associated to \eqref{Helm rectilinear} and obtained by
separating the variables. In particular, by setting \begin{equation}
\label{n* lambda q} n_*=\max_{\RR}{n},\quad q(x) = k^2 [ n_*^2 -
n(x)^2 ],
\end{equation} $e(x,\gamma_l)$ is the only $C^1$ solution of
\begin{equation*}
e'' + [\gamma_l - q(x)] e = 0, \quad \textmd{in } \RR,
\end{equation*}
such that $\|e(\cdot,\gamma_l)\|_{L^2(\RR)} = 1$ and which vanishes
exponentially as $|x|\to +\infty$ (see \cite{MS}).
$G^g= \sum_{l=1}^M G_l$ represents the guided part of the Green's
function, which involves the guided modes, i.e. the modes
propagating mostly inside the waveguide; each $G_l,\: l=1,\ldots,M,$
corresponds to a single guided mode.

In \eqref{G G^rad+somma G_l^g}, $G_0$ is the part of the Green's
function corresponding to the non-guided energy, i.e. the energy
radiated outside the waveguide. In \cite{CM2} the case
$n_{cl}:=n_+=n_-$ has been carefully studied; in particular, it was
proved that, for $\xi$ and $\zi$ fixed, the following asymptotic
expansions
\begin{equation*} 
G_0=\OO(R^{-\frac{1}{2}}),\quad \frac{\de G_0}{\de \nu} - ik n_{cl}
G_0=\OO(R^{-\frac{3}{2}}),
\end{equation*}
uniformly as $R\to+\infty$ on the sets $\de \Omega_R$, given by
\eqref{omega rho}.

In the present paper, since we are allowing $n_+$ and $n_-$ to be
different, we shall make use of the following result:

\begin{lemma} \label{lemma Grad asintotica}
Let $G_0$ be the Green's function mentioned above. Then, for $\xi$
and $\zi$ fixed, we have
\begin{equation}\label{Grad asintotica}
\ints_{\de \Omega_R} |G_0|^2 d\sigma =\OO(1),\quad \ints_{\de
\Omega_R} \Big{|} \frac{\de G_0}{\de \nu} - ik n(x) G_0 \Big{|}^2
d\sigma=\OO(R^{-1}),
\end{equation}
where $\Omega_R$ is given by \eqref{omega rho}.
\end{lemma}

\begin{proof}
The results are a consequence of the (uniform) asymptotic expansion
of $G_0$ for $R$ large. That can be done by following \cite{Ch} and
\cite{CM2}.
\end{proof}

\begin{lemma} \label{lemma singolarita}
Let $(x,z),(\xi,\zi) \in \RR^2$ and $\omega=(x-\xi,z-\zi)$ with
$|\omega|\leq 1$. There exists a positive constant $C_1$ independent
on $x,z,\xi,\zi,$ such that
\begin{equation}
\big{|} G_0(x,z;\xi,\zi) - \frac{1}{2\pi} \log |\omega| \big{|} \leq
  C_1. \label{Green singularity}
\end{equation}
\end{lemma}

\begin{proof}
In order to avoid heavy calculations, we carry out the scheme of the
proof only for the case studied in \cite{CM2}, i.e. for $n_{cl} :=
n_+ = n_- $.

Instead of proving \eqref{Green singularity}, we shall prove that
$\big{|} G_0(x,z;\xi,\zi) - G_{FS}(x,z;\xi,\zi) \big{|}$ is
uniformly bounded; here, we denoted by $G_{FS}$ the outgoing Green's
function of the free-space case, i.e. $\displaystyle G_{FS}
(x,z;\xi,\zi) = (4\pi i)^{-1} H_0^{(1)} (k n_{cl} |\omega|),$ where
$H_0^{(1)}$ is the zeroth-order Hankel function of the first kind.

The following integral representations will be useful for proving
the lemma:
\begin{eqnarray}
&& G_{FS}(x,z;\xi,\zi) = \frac{1}{4\pi i} \ints_{\Cam} e^{ikn_{cl}
[(x-\xi) \sin{t} + |z-\zi| \cos{t}]} dt, \label{G_fs represent} \\
&& G_0(x,z;\xi,\zi) = \ints_{\Cam} g(x,\xi;t) e^{ikn_{cl} ( [x]_h
\sin{t} + |z-\zi| \cos{t})} dt, \label{G_0 represent}
\end{eqnarray}
with $\Cam$ being the contour path shown in Fig.\ref{Fig Cam}. We
shall not write the explicit expression of $g$ in \eqref{G_0
represent} and we refer to formula (3.8) in \cite{CM2} for details
since, here, we will make use only of the following asymptotic
formula:
\begin{equation*}
g(x,\xi;t)= \frac{1}{4\pi i} \: e^{i kn_{cl} ( \{x\}_h - \xi ) \sin
t} \bigg\{ 1+ \frac{i}{2kn_{cl} \sin t} \ints_{\{\xi\}_h}^{\{x\}_h}
[d^2 - q(y)] dy \bigg\} + \OO \bigg(\frac{1}{|\sin t|^2} \bigg),
\end{equation*}
which holds as $|t| \to \infty$ on $\Cam$, uniformly for $x \in \RR$
and $\xi$ bounded (see Lemma A.2 in \cite{CM2}); here, $\{x\}_h :=
x-[x]_h$, with $[x]_h$ defined by \eqref{xh quadra}. Lemma A.3 in
\cite{CM2} assures that $g$ is bounded on $\Cam$. Thus, \eqref{Green
singularity} follows straightforwardly from \eqref{G_fs
represent},\eqref{G_0 represent}, the above asymptotic expansion of
$g$ and by observing that $\{x\}_h + [x]_h = x$.
\begin{figure}
\centering
\includegraphics[width=0.3\textwidth]{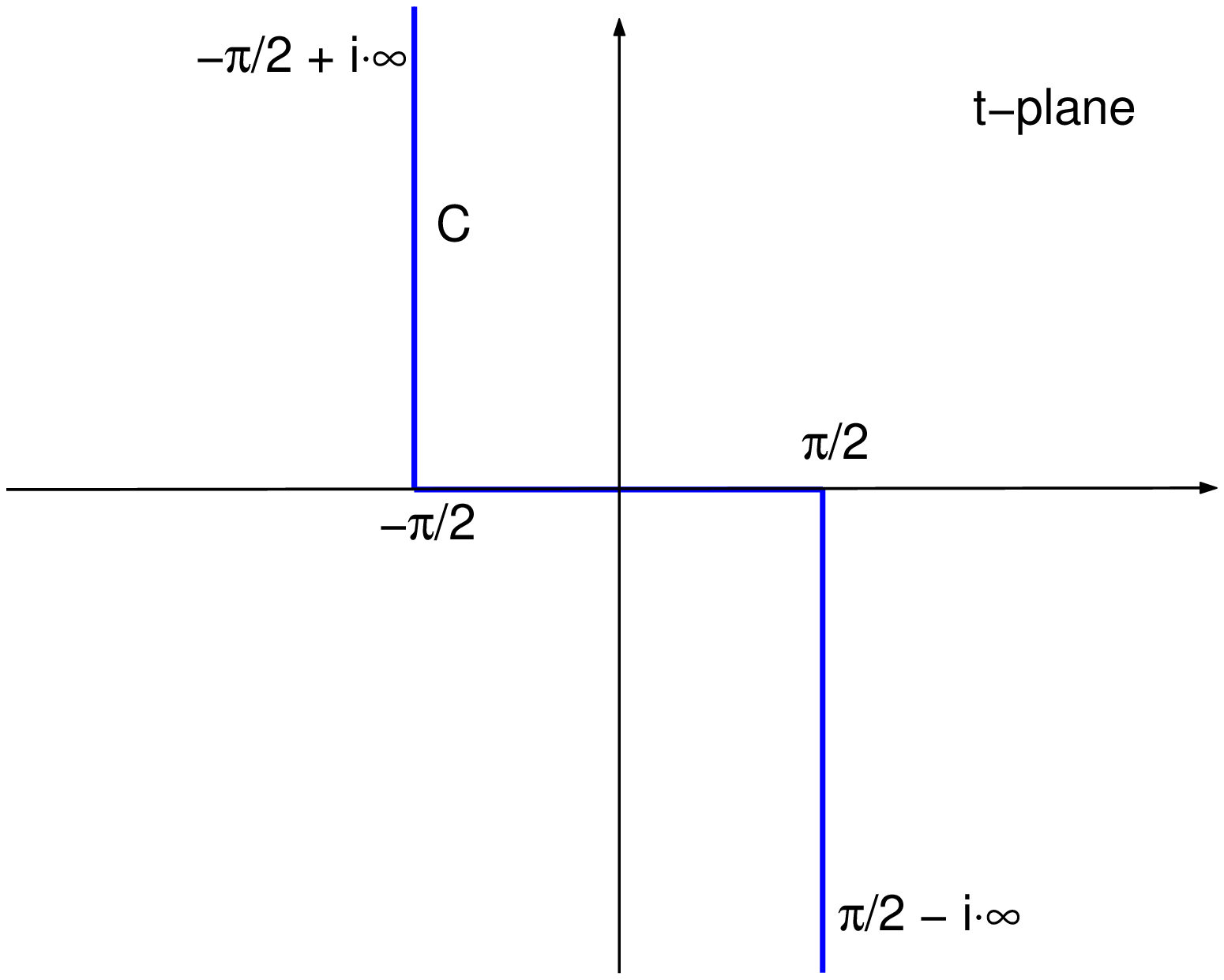}
\caption{The contour $\Cam$.} \label{Fig Cam}
\end{figure}
\end{proof}

\section{Proof of Theorem \ref{teo1}} \label{section uniqueness}
We consider a solution $u$ of \eqref{helm perturbata} and define
\begin{subequations} \label{u0 ul def}
\begin{equation}\label{u_l}
u_l(x,z)= e(x,\gamma_l) \ints_{-\infty}^{+\infty} u(\xi,z)
e(\xi,\gamma_l) d\xi, \quad l=1,\ldots,M,
\end{equation}
and
\begin{equation}\label{u_0}
    u_0 (x,z)= u(x,z) - \sum\limits_{l=1}^{M}  u_l(x,z).
\end{equation}
\end{subequations}

\begin{lemma} \label{lemma ul soluzioni di}
Let $u$ be a weak solution of
\begin{equation}\label{helm perturbata omogenea}
\Delta u + [k^2 n(x)^2 + p(x,z)] u = 0, \quad (x,z) \in \RR^2,
\end{equation}
and define $u_l, l=0,1,\ldots,M$, as in \eqref{u0 ul def}. Then,
$u_l$ is a weak solution of
\begin{equation*}
\Delta u_l + k^2 n(x)^2 u_l = -\psi_l, \quad l=0,1,\ldots,M,
\end{equation*}
where we set
\begin{subequations} \label{psi}
\begin{equation}
\psi_l (x,z) = e(x,\gamma_l) \ints_{\RR} p(\xi,z) u(\xi,z)
e(\xi,\gamma_l) d\xi, \quad l=1,\ldots,M,
\end{equation}
and
\begin{equation}
\psi_0 = p u - \sum\limits_{l=1}^M \psi_l.
\end{equation}
\end{subequations}
\end{lemma}

\begin{proof}
The proof is analogous to part of the proof of Theorem 2.6 in
\cite{CM2} and hence is omitted.
\end{proof}

\begin{lemma} \label{lemma rappresentazioni}
Let $(\xi,\zi) \in \RR^2$ be fixed and $R$ be such that $(\xi,\zi)
\in \Omega_R$. Let $u$ be a solution of \eqref{helm perturbata
omogenea}; then, we have the following identities:
\begin{subequations} \label{rappr u0 ul}
\begin{equation} \label{rappr u0}
u_0(\xi,\zi) + \ints_{\Omega_R} G_0 \psi_0 dx dz = \ints_{\de
\Omega_R} \left(u_0 \frac{\de G_0}{\de \nu} - G_0 \frac{\de u_0}{\de
\nu} \right) d\ell ,
\end{equation}
and
\begin{equation} \label{rappr u l}
e(\xi,\gamma_l) \ints_{-R}^R e(s,\gamma_l) u(s,\zi) ds  +
\ints_{Q_R} G_l \psi_l dx dz = \ints_{\de Q_R} \left(u_l \frac{\de
G_l}{\de \nu} - G_l \frac{\de u_l}{\de \nu} \right) d\ell,
\end{equation}
\end{subequations}
for $l=1,\ldots,M$, and where $\psi_l$, $l=0,\ldots,M$, are given by
\eqref{psi}.
\end{lemma}

\begin{proof} Thanks to Lemma \ref{lemma ul soluzioni di} we have
that
\begin{equation*}
\ints_{D} (u_l \Delta G_l - G_l \Delta u_l) dx dz = \ints_{D} \big[
u_l (\Delta G_l + k^2n(x)^2 G_l) + G_l \psi_l \big] dx dz,
\end{equation*}
for $l=0,1,\ldots,M$, and where $D$ is a (smooth enough) bounded
domain (notice that, since $u$ is a weak solution of \eqref{helm
perturbata omogenea}, by Theorem 8.8 in \cite{GT}, the integrals in
\eqref{identita integrali volume} make sense). The above formula and
the second Green's identity yield
\begin{equation} \label{identita integrali volume}
\ints_{\partial D} \left(u_l \frac{\partial G_l}{\partial \nu} - G_l
\frac{\partial u_l}{\partial \nu} \right) d\ell = \ints_{D} \big[
u_l (\Delta G_l + k^2n(x)^2 G_l) + G_l \psi_l \big] dx dz,
\end{equation}
for $l=0,1,\ldots,M$.

Firstly, consider the case $l=0$. Thanks to Lemma \ref{lemma
singolarita}, we know that $G_0$ has a singularity for
$(x,z)\equiv(\xi,\zi)$. We denote by $B_\ep$ the ball centered in
$(\xi,\zi)$ of radius $\ep$ and consider \eqref{identita integrali
volume} with $D=\Omega_R \setminus B_\ep$; thus, \eqref{rappr u0}
follows from \eqref{identita integrali volume}, Lemma \ref{lemma
singolarita} and by taking the limit as $\ep \to 0^+$.

Now, let $l=1,\ldots,M$ be fixed. From \eqref{G l guid} it follows
that $\Delta G_l + k^2n(x)^2 G_l$ has a singularity for $z=\zi$. By
setting $D=Q_R \setminus \{(x,z)\in\RR^2: |z-\zi|<\ep\}$ in
\eqref{identita integrali volume}, we obtain \eqref{rappr u l} by
taking the limit as $\ep\to 0^+$.
\end{proof}

\begin{proofTeo1}
Let assume that $u^1$ and $u^2$ are two bounded solutions of
\eqref{helm perturbata} satisfying \eqref{rad cond} and consider
$u=u^1-u^2$. It is clear that $u$ is a bounded solution of
\eqref{helm perturbata omogenea} and satisfies \eqref{rad cond}. We
write $u=u_0+u_1+\ldots+u_M$ as done in \eqref{u0 ul def}.

Let $(\xi,\zi)$ be fixed and consider $R$ large enough such that
$(\xi,\zi) \in \Omega_R$. We set
\begin{equation*}\Omega_R^{(l)} =
\begin{cases}
\Omega_R, & l=0, \\
Q_R, & l=1,\ldots,M,
\end{cases}
\end{equation*}
and
\begin{equation*}
J(R)=u_0(\xi,\zi) + \sum_{l=1}^M e(\xi,\gamma_l) \ints_{-R}^R
e(s,\gamma_l) u(s,\zi) ds + \sum_{l=0}^M \ints_{\Omega_R^{(l)}} G_l
\psi_l dx dz.
\end{equation*}
By summing up identities \eqref{rappr u0 ul} for $l=0,1,\ldots,M$
and thanks to a simple manipulation, we obtain that
\begin{equation*}
J(R) = \sum_{l=0}^M \ints_{\de \Omega_R^{(l)}} \left[u_l \left(
\frac{\de G_l}{\de \nu} - i\beta_l G_l\right) - G_l \left( \frac{\de
u_l}{\de \nu} - i\beta_l u_l \right) \right] d\ell,
\end{equation*}
where we set $\beta_0:=kn(x)$ (since it is not relevant in this
proof, we are omitting the dependence of $\beta_0$ on $x$).
Triangular and Cauchy-Schwartz inequalities yield
\begin{multline} \label{ineq in teo}
|J(R)| \leq \sum_{l=0}^M  \Bigg( \ints_{\de \Omega_R^{(l)}} |u_l|^2
d\ell \Bigg)^{\frac{1}{2}} \Bigg( \ints_{\de \Omega_R^{(l)}}
\bigg{|}\frac{\de
G_l}{\de \nu} - i\beta_l G_l\bigg{|}^2 d\ell \Bigg)^{\frac{1}{2}} \\
+ \Bigg( \ints_{\de \Omega_R^{(l)}} |G_l|^2 d\ell
\Bigg)^{\frac{1}{2}} \Bigg( \ints_{\de \Omega_R^{(l)}}
\bigg{|}\frac{\de u_l}{\de \nu} - i\beta_l u_l\bigg{|}^2 d\ell
\Bigg)^{\frac{1}{2}}.
\end{multline}

Thanks to \eqref{u0 ul def}, Lemma 3.7 in \cite{CM2} and
Fubini-Tonelli's theorem, we obtain that
\begin{equation} \label{limite J(R)}
\lim_{R\to +\infty} J(R) = u(\xi,\zi) + \ints_{\RR^2} G(x,z;\xi,\zi)
p(x,z) u (x,z) dx dz.
\end{equation}
From \eqref{G l guid}, \eqref{Grad asintotica} and since each $u_l,
l=0,1,\ldots,M,$ is bounded, we have that
\begin{equation*}
\ints_{\de \Omega_R^{(0)}} |G_0|^2 d\ell = \OO(1), \quad  \
\ints_{\de \Omega_R^{(l)}} |G_l|^2 d\ell = \OO(R), \ \ l=1,\ldots,M,
\end{equation*}
and
\begin{equation*}
\ints_{\de \Omega_R^{(l)}} |u_l|^2 d\ell = \OO(R),\quad
l=0,1,\ldots,M,
\end{equation*}
as $R\to +\infty$; furthermore, from \eqref{G l guid} we easily get
that
\begin{equation*}
\ints_{\de \Omega_R^{(l)}} \bigg{|}\frac{\de G_l}{\de \nu} -
i\beta_l G_l\bigg{|}^2 d\ell, \quad l=1,\ldots,M,
\end{equation*}
vanishes exponentially as $R\to +\infty$. From the above asymptotic
estimates, \eqref{Grad asintotica} and since $u$ satisfies
\eqref{rad cond}, it follows that the right hand side of \eqref{ineq
in teo} vanishes as $R\to +\infty$. Thus, by taking the limit for
$R\to +\infty$ in \eqref{ineq in teo}, from \eqref{limite J(R)} we
have that
\begin{equation*}
u(\xi,\zi) + \ints_{\RR^2} G(x,z;\xi,\zi) p(x,z) u (x,z) dx dz = 0.
\end{equation*}
Since $u$ is bounded and by setting $L=\sup\limits_{(x,z) \in
\RR^2}|u(x,z)|$, from the above formula we have
\begin{equation*}
L \leq L \sup_{(\xi,\zi) \in \RR^2} \ints_{\RR^2} |G(x,z;\xi,\zi)
p(x,z)| dx dz,
\end{equation*}
which, together with \eqref{p condizioni}, implies that $L=0$, i.e.
$u_1=u_2$.
\end{proofTeo1}

\section{Proof of Theorem \ref{teo2}} \label{section existence}
The proof of Theorem \ref{teo2} is a consequence of the following
two lemmas.

\begin{lemma} \label{lemma eidus}
Let $\vp \in L^2(\RR^2)$ be a complex valued function satisfying
(H1) and (H3). Then, the function
\begin{equation*}
w_0(x,z)= \ints_{\RR^2} G_0(x,z;\xi,\zi) \vp(\xi,\zi) d\xi d\zi
\end{equation*}
satisfies
\begin{equation*}
\lim_{R\to +\infty}\ints_{\partial \Omega_R}
    \Big{|} \frac{\partial w_0}{\partial \nu} - ikn(x) w_0
    \Big{|}^2 d\ell = 0,
\end{equation*}
with $\Omega_R$ given by \eqref{omega rho}.
\end{lemma}

\begin{proof}
We set $d(x,z)=\sqrt{[x]_h^2 + z^2}$ and notice that $\de \Omega_R$,
$R>0$, are the level sets of $d$. Let $(x,z) \in \de \Omega_R$ and
set $\rho=R^s$, for some $0<s<1$. We have:
\begin{equation*}
\begin{split}
\frac{\partial w_0}{\partial \nu}& (x,z)  - ikn(x) w_0(x,z) =
\ints_{\RR^2} [\nabla G_0 \cdot \nabla d(x,z) - i kn(x) G_0]
\vp(\xi,\zi) d\xi d\zi \\
& = \ints_{\Omega_\rho} [\nabla G_0 \cdot \nabla d(x,z-\zi) - i
kn(x) G_0] \vp d\xi d\zi \\
& \quad + \ints_{\Omega_\rho} \nabla G_0 \cdot \nabla [d(x,z) -
d(x,z-\zi)] \vp d\xi d\zi \\ & \ \ \quad + \ints_{\Omega_\rho^c
\setminus B_1(x,z)}  [\nabla G_0 \cdot \nabla d(x,z) - i kn(x)
G_0] \vp d\xi d\zi \\
& \ \ \quad \quad + \ints_{B_1(x,z)} \nabla G_0 \cdot \nabla d(x,z)
\vp d\xi
d\zi - \ints_{B_1(x,z)} i kn(x) G_0 \vp d\xi d\zi \\
& = I_1+I_2+I_3+I_4-I_5.
\end{split}
\end{equation*}
Since the quantity $\nabla G_0 \cdot \nabla d(x,z-\zi) - i kn(x)
G_0$ depends on $x,\xi$ and $z-\zi$, from Lemma \ref{lemma Grad
asintotica}, we have that $|\nabla G_0 \cdot \nabla d(x,z-\zi) - i
kn(x) G_0| = \OO(R^{-\frac{3}{2}})$ uniformly for $\xi$ and $\zi$
bounded, and thus $|I_1|=\OO (R^{-\frac{3}{2}})$ (notice that
\eqref{int vp} implies that $\vp \in L^1(\RR^2)$).

From (H1) we can assume that $\xi$ is bounded and, from Lemma
\ref{lemma Grad asintotica}, we infer that $|\nabla G_0|$ and $G_0$
are bounded in $\Omega_\rho$ for $R$ large enough. Thus, since
\begin{equation*}
|\nabla d(x,z) - \nabla d(x,z-\zi)| = \OO \Big( \frac{\zi}{R} \Big),
\end{equation*}
we have that $|I_2|$ is estimated (up to a multiplicative constant)
by
\begin{equation*}
\frac{1}{R} \ints_{\Omega_R} |\zi| |\vp(\xi,\zi)| d\xi d\zi.
\end{equation*}
Coarea formula (notice that $|\nabla d|=1$) and H\"{o}lder
inequality yield
\begin{equation*}
\ints_{\Omega_\rho} |\zi| |\vp (\xi,\zi)| d\xi d\zi \leq
\ints_0^{\rho} r \ints_{\de \Omega_r} |\vp| d\ell \ dr \leq
\sqrt{2(\pi+h)} \ints_0^{\rho} r^{\frac{3}{2}} \bigg( \ints_{\de
\Omega_r} |\vp|^2 d\ell \bigg)^{\frac{1}{2}} dr,
\end{equation*}
and thus, from \eqref{int vp}, we obtain that $I_2 = \OO
(R^{s(1-\delta) - 1})$.

From Lemma \ref{lemma singolarita} and since $|\nabla d|=1$, we
obtain that (up to a multiplicative constant) $|I_3|$ is bounded by
\begin{equation*}
\ints_{\Omega_\rho^c} |\vp (\xi,\zi)| d\xi d\zi,
\end{equation*}
and thus $I_3= \OO(R^{-s\delta})$.

In order to estimate $I_5$, we use H\"{o}lder inequality and notice
that $\|G_0\|_{L^2(B_1(x,z))}$ is bounded by a constant independent
on $\xi$ and $\zi$, as follows from \eqref{Green singularity}. Thus,
from coarea formula and \eqref{int vp}, we have
\begin{equation*}
\ints_{B_1(x,z)} |\vp(\xi,\zi)|^2 d\xi d\zi \leq \ints_{R-1}^{R+1}
\ints_{\de \Omega_r} |\vp|^2 d\ell \ dr \leq
2c_1(R-1)^{-(3+2\delta)},
\end{equation*}
which implies that $I_5= \OO (R^{-\frac{3}{2} - \delta})$.

By summing up the above estimates we find that
\begin{equation*}
|I_1+I_2+I_3+I_5| = \OO \left(\max
\{R^{-\frac{3}{2}},R^{s(1-\delta)-1},R^{-\delta s},R^{-\frac{3}{2} -
\delta}\}\right),
\end{equation*}
as $R\to +\infty$, and thus
\begin{equation*}
\ints_{\de \Omega_R} |I_1+I_2+I_3+I_5|^2 d\ell = \OO \left(\max
\{R^{-2},R^{2s(1-\delta)-1},R^{-2\delta
s+1},R^{-2(1+\delta)}\}\right),
\end{equation*}
as $R\to +\infty$. By choosing $0<\ep<1$ such that $\delta=\ep +
\frac{1}{2(1-\ep)}$ and setting $s=1-\ep$, we find that
\begin{equation*}
\lim_{R\to+\infty} \ints_{\de \Omega_R} |I_1+I_2+I_3+I_5|^2 d\ell
=0.
\end{equation*}

It remains to prove that
\begin{equation*}
\lim_{R\to+\infty} \ints_{\de \Omega_R} |I_4|^2 d\ell =0.
\end{equation*}
We set $\omega =(x-\xi,z-\zi)$. Working as in the proof of Lemma
\ref{lemma singolarita}, we prove that $|\nabla G_0 \cdot \omega |$
is bounded in $B_1(x,z)$ by a constant independent on $x,z,\xi,\zi$.
Thus, $|I_4|$ is estimated (up to a multiplicative constant) by
\begin{equation*}
 \ints_{B_1(x,z)} \frac{|\vp(\xi,\zi)|}{|\omega|}
d\xi d\zi,
\end{equation*}
where we set $p=(x-\xi,z-\zi)$. From H\"{o}lder inequality, we
estimate $|I_4|^2$ by
\begin{equation*}
\ints_{B_1(x,z)} \frac{d\xi d\zi}{|\omega|^{\frac{3}{2}}}
\ints_{B_1(x,z)} \frac{|\vp(\xi,\zi)|^2}{|\omega|^{\frac{1}{2}}}
d\xi d\zi = 4\pi \ints_{B_1(x,z)}
\frac{|\vp(\xi,\zi)|^2}{|\omega|^{\frac{1}{2}}} d\xi d\zi.
\end{equation*}
Fubini-Tonelli's Theorem yields
\begin{equation*}
\begin{split}
\ints_{\de \Omega_R} \ints_{B_1(x,z)}
\frac{|\vp(\xi,\zi)|^2}{|\omega|^{\frac{1}{2}}} d\xi d\zi\ dx dz  &
\leq \ints_{\de \Omega_R} \bigg( \ints_{\Omega_{R+1} \setminus
\Omega_{R-1}} \frac{|\vp(\xi,\zi)|^2}{|\omega|^{\frac{1}{2}}} d\xi
d\zi\ \bigg) dx dz
\\ &  = \ints_{\Omega_{R+1}
\setminus \Omega_{R-1}} |\vp(\xi,\zi)|^2 \bigg( \ints_{\de \Omega_R}
\frac{1}{|\omega|^{\frac{1}{2}}} dx dz \bigg)  d\xi d\zi.
\end{split}
\end{equation*}
Since
\begin{equation*}
\ints_{\de \Omega_R} \frac{1}{|\omega|^{\frac{1}{2}}} dx dz =
\OO(R),
\end{equation*}
and from \eqref{int vp}, we obtain that
\begin{equation*}
\ints_{\de \Omega_R} |I_4|^2 dx dz = \OO(R^{-2-2\delta}),
\end{equation*}
which completes the proof.
\end{proof}

\begin{lemma} \label{lemma eidus guidati}
Let $\vp$ be as in Lemma \ref{lemma eidus}. Then,
\begin{equation*}
w_l(x,z)= \ints_{\RR^2} G_l(x,z;\xi,\zi) \vp(\xi,\zi) d\xi d\zi,
\end{equation*}
$l=1,\ldots,M$, satisfies
\begin{equation*}
\lim_{R\to +\infty} \sqrt{R} \ints_{Q_R}
    \Big{|} \frac{\partial w_l}{\partial \nu} - i\beta_l w_l
    \Big{|}^2 d\ell = 0.
\end{equation*}
\end{lemma}

\begin{proof}
Let $(x,z)\in \de Q_R$, with $|x|=R$. Thus, $\frac{\de}{\de \nu} =
\frac{\de}{\de |x|}$ and it is easy to show that
\begin{equation*}
\begin{split}
\Big{|} \frac{\partial w_l}{\partial \nu} - i\beta_l w_l \Big{|} & =
K_l |e(x,\gamma_l)| \ints_{\RR^2} |e(\xi,\gamma_l) \vp(\xi,\zi)|
d\xi d\zi \\ & \leq K_l |e(x,\gamma_l)|
\|e(\cdot,\gamma_l)\|_{L^\infty(\RR)} \|\vp\|_{L^1(\RR^2)},
\end{split}
\end{equation*}
with $K_l,\: l=1,\ldots,N,$ positive constants; since
$|e(x,\gamma_l)|$ vanishes exponentially as $|x|\to +\infty$, we
obtain that
\begin{equation*}
\lim_{R\to +\infty} \sqrt{R} \ints_{Q_R\cap \{|x|=R\}} \Big{|}
\frac{\partial w_l}{\partial \nu} - i\beta_l w_l \Big{|}^2 d\ell =
0.
\end{equation*}

Now, we consider $(x,z)\in \de Q_R$, with $|z|=R$ (thus
$\frac{\de}{\de \nu} = \frac{\de}{\de |z|}$). We write
\begin{equation*}
\frac{\partial w_l}{\partial \nu} - i\beta_l w_l = \ints_{\{|\zi|<
R\}} + \ints_{\{|\zi|\geq R\}} \left[ \frac{\de
G_l(x,z;\xi,\zi)}{\de |z|} - i \beta_l G_l(x,z;\xi,\zi) \right]
\vp(\xi,\zi) d\xi d\zi.
\end{equation*}
From \eqref{G l guid} it follows that the first integral on the
right hand side vanishes, since there $|z|>|\zi|$. We estimate the
second integral as follows:
\begin{equation*}
\ints_{\{|\zi|\geq R\}} \bigg{|} \frac{\de G_l}{\de |z|} - i \beta_l
G_l \bigg{|} |\vp| d\xi d\zi \leq
\|e(\cdot,\gamma_l)\|_{L^\infty(\RR)} |e(x,\gamma_l)| \ints_{\RR^2
\setminus \Omega_R} |\vp(\xi,\zi)| d\xi d\zi.
\end{equation*}
Since $\vp$ satisfies (H1), coarea formula and H\"{o}lder inequality
yield
\begin{equation*}
\ints_{\RR^2 \setminus \Omega_R} |\vp(\xi,\zi)| d\xi d\zi =
\ints_R^{+\infty} \ints_{\de \Omega_r} |\vp| d\ell dr \leq
\sqrt{2\pi} \ints_R^{+\infty} \sqrt{r} \left( \ints_{\de \Omega_r}
|\vp|^2 d\ell \right)^{\frac{1}{2}} dr;
\end{equation*}
from \eqref{int vp} we obtain that
\begin{equation*}
\ints_R^{+\infty} \sqrt{r} \left(\  \ints_{\de \Omega_r} |\vp|^2
d\ell \right)^{\frac{1}{2}} dr \leq \frac{\sqrt{c_1}}{\delta}
R^{-\delta},
\end{equation*}
and then
\begin{equation*}
\bigg{|} \frac{\partial w_l}{\partial \nu} - i\beta_l w_l \bigg{|}
\leq \frac{\sqrt{2\pi c_1}}{\delta} \,
\|e(\cdot,\gamma_l)\|_{L^\infty(\RR)} |e(x,\gamma_l)|
 R^{-\delta}.
\end{equation*}
Since $\|e(\cdot,\gamma_l)\|_2 = 1$, we have that
\begin{equation*}
\begin{split}
\ints_{Q_R\cap \{|z|=R\}} \Big{|} \frac{\partial w_l}{\partial \nu}
- i\beta_l w_l \Big{|}^2 d\ell & \leq \frac{2\pi c_1}{\delta^2}
\|e(\cdot,\gamma_l)\|_{L^\infty(\RR)} R^{-2\delta} \ints_{-R}^R
|e(x,\gamma_l)|^2 dx \\
& \leq \frac{2\pi c_1}{\delta^2}
\|e(\cdot,\gamma_l)\|_{L^\infty(\RR)} R^{-2\delta};
\end{split}
\end{equation*}
from the above estimate and since $\delta> \frac{1}{2}$, we obtain
that
\begin{equation*}
\lim_{R\to +\infty} \sqrt{R} \ints_{Q_R\cap \{|z|=R\}} \Big{|}
\frac{\partial w_l}{\partial \nu} - i\beta_l w_l \Big{|}^2 d\ell =
0,
\end{equation*}
which completes the proof.
\end{proof}

\begin{proofTeo2} Firstly
we prove that $u$ is bounded and then show that it satisfies
\eqref{rad cond}.

We notice that, if we prove that $\ints_{\RR^2} Gf$ is bounded, then
we conclude that $u$ is bounded, as follows from \eqref{p
condizioni} and a contraction mapping theorem. In order to prove
that, we write
\begin{equation}\label{int in dim esistenza}
\ints_{\RR^2} G(x,z;\xi,\zi) f(\xi,\zi) d\xi d\zi = \ints_{B_1(x,z)}
+ \ints_{\RR^2\setminus B_1(x,z)} G(x,z;\xi,\zi) f(\xi,\zi) d\xi
d\zi.
\end{equation}
H\"{o}lder inequality and \eqref{Green singularity} imply that the
first integral on the right hand side is bounded by the $L^2$ norm
of $f$ multiplied by a constant independent on $(x,z)$. We notice
that the assumptions on $f$ imply that $f \in \L^1(\RR^2)$; since
$G$ is bounded outside $B_1(x,z)$ (as follows from Lemmas \ref{lemma
Grad asintotica} and \ref{lemma singolarita}), we obtain the
boundness of the second integral on the right hand side in
\eqref{int in dim esistenza} and conclude that $u$ is bounded.

It remains to prove that $u$ satisfies the radiation condition
\eqref{rad cond}. We write \eqref{eq integrale u} as
\begin{equation*}
u(x,z) = \sum_{l=0}^M \ints_{\RR^2} G_l(x,z;\xi,\zi)
[f(\xi,\zi)-p(\xi,\zi)u(\xi,\zi)] d\xi d\zi;
\end{equation*}
since $u$ is bounded, the conclusion follows straightforwardly from
the assumptions on $f$ and $p$ by using Lemmas \ref{lemma eidus} and
\ref{lemma eidus guidati}.
\end{proofTeo2}

\begin{remark}
In Theorem \ref{teo2}, we assumed that $f$ and $p$ satisfy (H1). Such an assumption is
due to the fact that, for proving Lemma \ref{lemma eidus}, we need an uniform asymptotic expansion of the
far-field of $G_0$, that we indeed have only if we assume that $\xi$ is bounded. We notice that, in Lemma
\ref{lemma eidus guidati} the assumption can be omitted, as it is
clear from its proof.
\end{remark}


\begin{thebibliography}{Z-Z-Z}
\bibitem[Ch]{Ch} {\sc T.~Christiansen}, \emph{Scattering theory for
perturbed stratified media}, Journal d'Analyse Math\'{e}matique, 76
(1998), pp. 1--44.

\bibitem[Ci1]{Ci1} {\sc G.~Ciraolo}, \emph{Non-rectilinear waveguides:
analytical and numerical results based on the Green's function},
Ph.D. Thesis, \verb+http://www.math.unipa.it/~g.ciraolo/+ .

\bibitem[Ci2]{Ci2} {\sc G.~Ciraolo}, \emph{A method of variation of boundaries for waveguide grating couplers},
Applicable Analysis, 87 (2008), no.9, pp. 1019-1040.

\bibitem[CM1]{CM1} {\sc G.~Ciraolo and R.~Magnanini}, \emph{Analytical results for 2-D non-rectilinear
waveguides based on the Green's function}, Math. Methods Appl. Sci.,
31 (2008), no.13, pp. 1587-1606.

\bibitem[CM2]{CM2} {\sc G.~Ciraolo -- R.~Magnanini}, \emph{A radiation condition for uniqueness in a
wave propagation problem for 2-D open waveguides}, Math. Methods
Appl. Sci., ?? (2009), no.??, pp. ??.


\bibitem[Ei]{Ei} {\sc D.~Eidus}, \emph{The principle of limiting absorption},
Thirteen papers on functional analysis and partial differential
equations, American Mathematical Society Translations, Series 2, 47
(1965), pp. 157 -- 191.

\bibitem[GT]{GT} {\sc D.~Gilbarg and N.~S. Trudinger}, \emph{Elliptic partial differential equations of second
order}. Springer-Verlag, 1983.

\bibitem[H\"{o}]{Ho} {\sc L.~H\"{o}rmander}, \emph{The analysis of linear partial differential operators. II.
Differential operators with constant coefficients.} Classics in
Mathematics. Springer-Verlag, Berlin, 2005.

\bibitem[KNH]{KNH} {\sc E.~M.~Kartchevski, A.~I.~Nosich and G.~W.~Hanson,}
\emph{Mathematical analysis of the generalized natural modes of an
inhomogeneous optical fiber}, SIAM J. Appl. Math., 65 (2005), no. 6,
pp. 2033 -- 2048.

\bibitem[Mi1]{Mi1} {\sc W.~L.~Miranker, } \emph{Uniqueness and
representation theorems for solutions of $\Delta u + k^2 u =0$ in
infinite domains}, Journal of Mathematics and Mechanics, 6 (1957),
pp. 847 -- 854.

\bibitem[Mi2]{Mi2} {\sc W.~L.~Miranker, } \emph{The reduced wave equation
in a medium with a variable index of refraction.} Comm. Pure Appl.
Math., 10 (1957) pp. 491--502.

\bibitem[MS]{MS} {\sc R.~Magnanini and F.~Santosa,
} \emph{Wave propagation in a 2-D optical waveguide}, SIAM J. Appl.
  Math., 61 (2001) 1237 -- 1252.

\bibitem[Rel]{Rel} {\sc F.~Rellich,} \emph{\"{U}ber das asymptotische
Verhalten del L\"{o}sungen von $\Delta u + \lam u =0$ in unendlichen
Gebieten}, Jahresbericht der Deutschen Mathematiker-Vereinigung, 53
(1943), pp. 57 -- 65.

\bibitem[So1]{So1} {\sc A.~Sommerfeld,} \emph{Die Greensche Funktion der
Schwingungsgleichung}, Deutsche Math.-Ver. 21 (1912), pp. 309-353.

\bibitem[So2]{So2} {\sc A.~Sommerfeld,} \emph{Partial differential equations in physics},
Academic Press, 1949.

\bibitem[SL]{SL} {\sc A.~W.~Snyder and D.~Love},
  \emph{Optical Waveguide Theory}, Chapman and Hall, London, 1974.

\bibitem[Ti]{Ti} {\sc E.~C.~Titchmarsh},
  \emph{Eigenfunction expansions associated with second-order differential
  equations}. Oxford at the Clarendon Press, Oxford, 1946.

\bibitem[We]{We} {\sc R.~Weder}, \emph{Spectral and scattering theory for wave propagation
in perturbed stratified media}. Applied Mathematical Sciences, 87.
Springer-Verlag, New York, 1991.


\bibitem[Wi]{Wi} {\sc C.~H.~Wilcox}, \emph{Sound propagation in stratified fluids},
Applied Mathematical Sciences, 50. Springer-Verlag, New York, 1984.

\bibitem[Xu1]{Xu1} {\sc Y.~Xu}, \emph{Scattering of acoustic waves by
an obstacle in a stratified medium}. Partial differential equations
with real analysis, Pitman Res. Notes Math. Ser., 263 (1992), pp.
147--168.

\bibitem[Xu2]{Xu2} {\sc Y.~Xu}, \emph{Radiation condition and scattering problem for
time-harmonic acoustic waves in a stratified medium with a
nonstratified inhomogeneity}, IMA J. Appl. Math. 54 (1995), no. 1,
pp. 9--29.

\end{thebibliography}
\end{document}